\documentclass[preprint,11pt,sort&compress,numbers]{elsarticle}

\usepackage{amsmath,amsthm,amsfonts,amssymb,latexsym,mathrsfs,url}

\newtheorem{thm}{Theorem}[section]
\newtheorem{lem}[thm]{Lemma}

\newtheorem{prop}[thm]{Proposition}
\newtheorem{ques}[thm]{Question}
\newtheorem{conj}[thm]{Conjecture}
\newtheorem*{SE-thm}{Schoenberg-Edrei Theorem}
\newtheorem*{PS-thm}{P\'olya-Schur Theorem}

\theoremstyle{definition}

\def\diag{\textrm{diag}}
\def\c{\ulcorner}

\numberwithin{equation}{section}

\begin{document}

\begin{frontmatter}

\title{Total positivity of Narayana matrices}

\author[a]{Yi Wang}
\ead{wangyi@dlut.edu.cn}
\author[b]{Arthur L.B. Yang}
\ead{yang@nankai.edu.cn}

\address[a]{School of Mathematical Sciences, Dalian University of Technology, Dalian 116024, P.R. China}
\address[b]{Center for Combinatorics, LPMC, Nankai University, Tianjin 300071, P.R. China}


\begin{abstract}
We prove the total positivity of the Narayana triangles of type $A$ and type $B$, and thus affirmatively confirm a conjecture of Chen, Liang and Wang and a conjecture of Pan and Zeng. We also prove the strict total positivity of the Narayana squares of type $A$ and type $B$.
\end{abstract}

\begin{keyword}
Totally positive matrices, the Narayana triangle of type $A$, the Narayana triangle of type $B$, the Narayana square of type $A$, the Narayana square of type $B$
\\[5pt]
\noindent \emph{AMS Classification 2010:} 05A10 \sep 05A20
\end{keyword}

\end{frontmatter}

\section{Introduction}

Let $M$ be a (finite or infinite) matrix of real numbers.
We say that $M$ is {\it totally positive} (TP)
if all its minors are nonnegative, and we say that it is {\it strictly totally positive} (STP)
if all its minors are positive.
Total positivity is an important and powerful concept
and arises often in analysis, algebra, statistics and probability,
as well as in combinatorics.
See \cite{And87,Bre95,Bre96,CLW-EuJC,CLW-LAA,FJ11,Kar68,Pin10} for instance.

Let $C(n,k)=\binom{n}{k}$.
It is well known~\cite[P. 137]{Kar68} that the Pascal triangle
$$P=\left[C(n,k)\right]_{n,k\geq 0}
=\left[\begin{array}{rrrrrr}
1 &  &  &  &  &  \\
1 & 1 &   &   &   &   \\
1 & 2 & 1 &   &   &   \\
1 & 3 & 3 & 1 &   &   \\
1 & 4 & 6 & 4 & 1 &   \\
\vdots &  & &  &  & \ddots \\
\end{array}\right]$$
is totally positive. Let
$$P^\c=\left[C({n+k},{k})\right]_{n,k\geq 0}=
\left[\begin{array}{ccccc}
1 & 1 & 1 & 1 & \cdots\\
1 & 2 & 3  & 4 & \\
1 & 3 & 6 & 10  & \\
1 & 4 & 10 & 20 & \\
\vdots & & & & \ddots \\
\end{array}\right]$$
be the Pascal square.
Then $P^\c=PP^T$ by the Vandermonde convolution formula
$$\binom{n+k}{k}=\sum_{i}\binom{n}{i}\binom{k}{i}.$$
Note that the transpose and the product of matrices preserve total positivity.
Hence $P^\c$ is also TP.

The main objective of this note is to prove the following two conjectures on the total positivity of the Narayana triangles.
Let $NA(n,k)=\frac{1}{k+1}\binom{n+1}{k}\binom{n}{k}$, which are commonly known as the Narayana numbers.
Let
$$N_A=\left[NA(n,k)\right]_{n,k\geq 0}
=\left[\begin{array}{rrrrrr}
1 &  &  &  &  &  \\
1 & 1 &   &   &   &   \\
1 & 3 & 1 &   &   &   \\
1 & 6 & 6 & 1 &   &   \\
1 & 10 & 20 & 10 & 1 &   \\
\vdots &  & &  &  & \ddots \\
\end{array}\right].$$
The Narayana numbers $NA(n,k)$ have many combinatorial interpretations. An interesting one is that
they appear as the rank numbers of the poset of noncrossing partitions associated to a Coxeter group of type $A$,
see Armstrong \cite[Chapter 4]{Armstrong09}. For this reason, we call $N_A$ the Narayana triangle of type $A$.
Chen, Liang and Wang \cite{CLW-LAA} proposed the following conjecture.
\begin{conj}[{\cite[Conjecture 3.3]{CLW-LAA}}]\label{CLW}
The Narayana triangle $N_A$ is TP.
\end{conj}

Let $NB(n,k)=\binom{n}{k}^2$, and let
$$N_B=\left[NB(n,k)\right]_{n,k\geq 0}
=\left[\begin{array}{rrrrrr}
1 &  &  &  &  &  \\
1 & 1 &   &   &   &   \\
1 & 4 & 1 &   &   &   \\
1 & 9 & 9 & 1 &   &   \\
1 & 16 & 36 & 16 & 1 &   \\
\vdots &  & &  &  & \ddots \\
\end{array}\right].$$
We call $N_B$ the Narayana triangle of type $B$ since the numbers
$NB(n,k)$ can be interpreted as the rank numbers of the poset of noncrossing partitions associated to a Coxeter group of type $B$,
see also Armstrong \cite[Chapter 4]{Armstrong09} and references therein.
Pan and Zeng \cite{PZ16} proposed the following conjecture.
\begin{conj}[{\cite[Conjecture 4.1]{PZ16}}]\label{PZ}
The Narayana triangle $N_B$ is  TP.
\end{conj}

In this note, we will prove that the Narayana triangles $N_A$ and $N_B$ are TP just like the Pascal triangle in a unified approach.
We also prove that the corresponding Narayana squares
$$
N_A^\c=\left[NA(n+k,k)\right]_{n,k\geq 0}
=\left[\begin{array}{rrrrrr}
1 & 1 & 1  & 1 &  &  \cdots \\
1 & 3 & 6  & 10     &   \\
1 & 6 & 20 & 50  &      \\
1 & 10 & 50 & 175  &      \\
\vdots &  & &  &   \ddots \\
\end{array}\right]$$
and
$$N_B^\c=\left[NB({n+k},{k})\right]_{n,k\geq 0}=
\left[\begin{array}{ccccc}
1 & 1 & 1 & 1 & \cdots\\
1 & 4 & 9  & 16 & \\
1 & 9 & 36 & 100  & \\
1 & 16 & 100 & 400 & \\
\vdots & & & & \ddots \\
\end{array}\right]$$
are STP, as well as the Pascal square.

\section{The Narayana triangles}

The main aim of this section is to prove the total positivity of the Narayana triangles $N_A$ and $N_B$.

Before proceeding to the proof, let us first note a simple property of totally positive matrices.
Let $X=[x_{n,k}]$ and $Y=[y_{n,k}]$ be two matrices.
If there exist positive numbers $a_n$ and $b_k$ such that $y_{n,k}=a_nb_kx_{n,k}$ for all $n$ and $k$,
then we denote $x_{n,k}\sim y_{n,k}$ and $X\sim Y$.
The following result is direct by definition.

\begin{prop}\label{prop-eq}
Suppose that $X\sim Y$.
Then the matrix $X$ is TP (resp. STP) if and only if the matrix $Y$ is TP (resp. STP).
\end{prop}

Our proof of Conjectures \ref{CLW} and \ref{PZ} is based on the P\'olya frequency property of certain sequences.
Let $(a_n)_{n\ge 0}$ be an infinite sequence of real numbers, and
define its Toeplitz matrix as
$$[a_{n-k}]_{n,k\ge 0}=
\left[\begin{array}{ccccc}
a_0 &  &  &  & \\
a_1 & a_0 &   &   & \\
a_2 & a_1 & a_0 &   & \\
a_3 & a_2 & a_1 & a_0 & \\
\vdots & & & & \ddots \\
\end{array}\right].$$
Recall that $(a_n)_{n\ge 0}$ is said to be a {\it P\'olya frequency} (PF) sequence
if its Toeplitz matrix is TP.
The following is the fundamental representation theorem for PF sequences,
see Karlin~\cite[p. 412]{Kar68} for instance.

\begin{SE-thm}\label{SE-thm}
A nonnegative sequence $(a_0=1,a_1,a_2,\ldots)$ is PF
if and only if its generating function has the form
$$\sum_{n\ge 0}a_nx^n=\frac{\prod_j(1+\alpha_j x)}{\prod_j(1-\beta_j x)}e^{\gamma x},$$
where $\alpha_j,\beta_j,\gamma\ge 0$ and $\sum_j(\alpha_j+\beta_j)<+\infty$.
\end{SE-thm}

Clearly, the sequence $(1/n!)_{n\geq 0}$ is PF by Schoenberg-Edrei Theorem,
which implies that the corresponding Toeplitz matrix $[a_{n-k}]=\left[1/(n-k)!\right]$ is TP.
Also, note that $$\binom{n}{k}=\frac{n!}{k!(n-k)!}\sim \frac{1}{(n-k)!}.$$
Hence the Pascal triangle $P$ is TP by Proposition~\ref{prop-eq}.

%
%
%

We are now in a position to prove Conjectures \ref{CLW} and \ref{PZ}.

\begin{thm}\label{thm-Narayana-triangle}
The Narayana triangles $N_A$ and $N_B$ are TP.
\end{thm}
\begin{proof}
We have
$$NA(n,k)=\frac{n!(n+1)!}{k!(k+1)!(n-k)!(n-k+1)!}\sim\frac{1}{(n-k)!(n-k+1)!}$$
and $$NB(n,k)=\frac{n!^2}{k!^2(n-k)!^2}\sim\frac{1}{(n-k)!^2}.$$
So, to show that the Narayana triangles $N_A$ and $N_B$ are TP,
it suffices to show that the sequences $(1/(n!(n+1)!))_{n\geq 0}$ and $(1/(n!^2))_{n\geq 0}$ are PF.
Based a classic result of Laguerre on multiplier sequences, Chen, Ren and Yang \cite[Proof of Conjecture 1.1]{CRY16} already proved that the sequence $(1/((t)_nn!))_{n\geq 0}$ is PF for any $t>0$, where
$(t)_n=t(t+1)\cdots(t+n-1)$. Letting $t=2$ (resp. $t=1$), we obtain the PF property of $(1/(n!(n+1)!))_{n\geq 0}$ (resp. $(1/(n!^2))_{n\geq 0}$), as desired.
\end{proof}

The method used here applies equally well to the triangle composed of $m$-Narayana numbers, which we will recall below.
Fix an integer $m\geq 0$. For any $n\geq m$ and $0\leq k\leq n-m$, the $m$-Narayana number $NA_{\langle m\rangle}(n,k)$
is given by
\begin{align}\label{eq-mnarayana}
NA_{\langle m\rangle}(n,k)=\frac{m+1}{n+2}\binom{n+2}{k+1}\binom{n-m}{k}.
\end{align}
When $m=0$ we get the usual Narayana numbers $NA(n,k)$.
For more information on the numbers $NA_{\langle m\rangle}(n,k)$,
see \cite{oeis}. It is easy to show that the Narayana triangle $N_A$ is symmetric: $NA(n,k)=NA(n,n-k)$, but
$$N_{A,{\langle m\rangle}}=\left[NA_{\langle m\rangle}(n,k)\right]_{n\geq m, 0\leq k\leq n-m}$$
and
$$\overleftarrow{N}_{A,{\langle m\rangle}}=\left[NA_{\langle m\rangle}(n,n-m-k)\right]_{n\geq m, 0\leq k\leq n-m}$$
are two different triangles for $m\geq 1$.
The proof of Theorem \ref{thm-Narayana-triangle} carries over directly to the following more general result.

\begin{thm}\label{thm-Narayana-triangle-m}
For any $m\geq 0$, both $N_{A,{\langle m\rangle}}$ and $\overleftarrow{N}_{A,{\langle m\rangle}}$ are TP.
\end{thm}


\section{The Narayana squares}

The object of this section is to prove the total positivity of the Narayana squares $N_A^\c$ and $N_B^\c$.
Our proof is based on the theory of Stieltjes moment sequences.

Given an infinite sequence $(a_n)_{n\ge 0}$ of real numbers, define its Hankel matrix as
$$[a_{n+k}]_{n,k\ge 0}=
\left[\begin{array}{ccccc}
a_0 & a_1 & a_2 & a_3 & \cdots\\
a_1 & a_2 & a_3 & a_4 & \\
a_2 & a_3 & a_4 & a_5 & \\
a_3 & a_4 & a_5 & a_6 & \\
\vdots & & & & \ddots \\
\end{array}\right].$$
We say that $(a_n)_{n\ge 0}$ is a {\it Stieltjes moment} (SM) sequence if it has the form
\begin{equation*}\label{i-e}
a_n=\int_0^{+\infty}x^nd\mu(x),
\end{equation*}
where $\mu$ is a non-negative measure on $[0,+\infty)$.
The following is a classic characterization for Stieltjes moment sequences
(see \cite[Theorem 4.4]{Pin10} for instance).

\begin{lem}\label{PSz}
A sequence $(a_n)_{n\ge 0}$ is SM if and only if 
\begin{enumerate}[\rm (i)]
\item the Hankel matrix $[a_{i+j}]$ is STP; or
\item both $[a_{i+j}]_{0\le i,j\le n}$ and $[a_{i+j+1}]_{0\le i,j\le n}$ are positive definite.
\end{enumerate}
\end{lem}

Many well-known counting coefficients are Stieltjes moment sequences, see \cite{LMW-DM}.
For example, the sequence $(n!)_{n\ge 0}$ is a Stieltjes moment sequence since
$$n!=\int_0^{+\infty}x^ne^{-x}dx=\int_0^{+\infty}x^nd\left(1-e^{-x}\right).$$
Thus the corresponding Hankel matrix $[(n+k)!]$ is STP.
Note that
$$\binom{n+k}{k}=\frac{(n+k)!}{n!k!}\sim (n+k)!.$$
Hence the Pascal square
$P^\c$
is also STP. The main result of this section is as follows.

\begin{thm}\label{thm-Narayana-square}
The Narayana squares
$N_A^\c$ and $N_B^\c$
are STP.
\end{thm}
\begin{proof}
We have
\begin{align*}
NA(n+k,k)&=\frac{(n+k)!(n+k+1)!}{k!(k+1)!n!(n+1)!}\sim (n+k)!(n+k+1)!
\end{align*}
and
\begin{align*}
NB(n+k,k)&=\frac{(n+k)!^2}{n!^2k!^2}\sim (n+k)!^2.
\end{align*}
So, to show that the Narayana squares $N_A^\c$ and $N_B^\c$ are STP,
it suffices to show that the sequences $(n!(n+1)!)_{n\ge 0}$ and $((n!)^2)_{n\ge 0}$ are SM.

Note that the submatrix of a STP matrix is still STP.
Hence if the sequence $(a_n)_{n\ge 0}$ is SM,
then so is its shifted sequence $(a_{n+1})_{n\ge 0}$
by Lemma \ref{PSz} (i).
Now the sequence $(n!)_{n\ge 0}$ is SM,
so is the sequence $((n+1)!)_{n\ge 0}$.
On the other hand,
the famous Schur product theorem states that
the Hadamard product $[a_{i,j}b_{i,j}]$ of two positive definite matrices $[a_{i,j}]$ and $[b_{i,j}]$ is still positive definite.
As a result, if both $(a_n)_{n\ge 0}$ and $(b_n)_{n\ge 0}$ are SM,
then so is $(a_nb_n)_{n\ge 0}$
by Lemma \ref{PSz} (ii).
We refer the reader to \cite[\S 4.10.4]{Pin10} for details.
Thus we conclude that both $(n!(n+1)!)_{n\ge 0}$ and $((n!)^2)_{n\ge 0}$ are SM,
as required.
\end{proof}

We can also consider the strict total positivity of the $m$-th Narayana square:
$$
N_{A,{\langle m\rangle}}^{\c}=\left[NA_{\langle m\rangle}(n+k,k)\right]_{n\geq m, k\geq 0},$$
where $NA_{\langle m\rangle}(n,k)$ is given by \eqref{eq-mnarayana}.
The following result can be proved in the same way as above.

\begin{thm}\label{thm-Narayana-square-m}
For any $m\geq 0$, the square $N_{A,{\langle m\rangle}}^{\c}$ is STP.
\end{thm}


\section{Remarks}


There are various generalizations of classical Narayana numbers, see for instance \cite{Armstrong09, Barry11, CYZ1601, CYZ1602, Petersen15}.
As we mentioned before, the numbers $NA(n,k)$ (resp. $NB(n,k)$) appear as the rank numbers of the poset of generalized noncrossing partitions associated to a Coxeter group of type $A$ (resp. $B$). These posets are further generalized by
Armstrong \cite{Armstrong09} by introducing the notion of $m$-divisible noncrossing partitions for any positive integer $m$
and any finite Coxeter group. Armstrong also showed that these generalized posets are not lattices but are still graded.

Fixing an integer $m\geq 1$, for $n\geq k\geq 0$ set
\begin{align*}
FNA_{\langle m\rangle}(n,k)&=\frac{1}{n+1}\binom{n+1}{k}\binom{m(n+1)}{n-k}\\
FNB_{\langle m\rangle}(n,k)&=\binom{n}{k}\binom{mn}{n-k}.
\end{align*}
These numbers are called the Fuss-Narayana numbers by Armstrong \cite{Armstrong09}, who proved that
$FNA_{\langle m\rangle}(n,k)$ (resp. $FNB_{\langle m\rangle}(n,k)$)
are the rank numbers of the poset of $m$-divisible noncrossing partitions associated to a Coxeter group of type $A$ (resp. $B$).

Note that,  for any $m\geq 2$, we have
$$FNA_{\langle m\rangle}(n,k)\neq FNA_{\langle m\rangle}(n,n-k), FNB_{\langle m\rangle}(n,k)\neq FNB_{\langle m\rangle}(n,n-k).$$
Now define the Fuss-Narayana triangles
\begin{align*}
FN_{A,{\langle m\rangle}}=\left[FNA_{\langle m\rangle}(n,k)\right]_{n,k\geq 0} ,\quad \overleftarrow{FN}_{A,{\langle m\rangle}}=\left[FNA_{\langle m\rangle}(n,n-k)\right]_{n,k\geq 0},\\
FN_{B,{\langle m\rangle}}=\left[FNB_{\langle m\rangle}(n,k)\right]_{n,k\geq 0} ,\quad \overleftarrow{FN}_{B,{\langle m\rangle}}=\left[FNB_{\langle m\rangle}(n,n-k)\right]_{n,k\geq 0}
\end{align*}
and the Fuss-Narayana squares
\begin{align*}
FN_{A,{\langle m\rangle}}^{\c}=\left[FNA_{\langle m\rangle}(n+k,k)\right]_{n,k\geq 0},\\
FN_{B,{\langle m\rangle}}^{\c}=\left[FNB_{\langle m\rangle}(n+k,k)\right]_{n,k\geq 0}.
\end{align*}
We proposed the following conjecture.
\begin{conj} For any $m\geq 1$, the Fuss-Narayana triangles are TP and the Fuss-Narayana squares are STP.
\end{conj}


There are other symmetric combinatorial triangles,
which are TP and the corresponding squares are STP.
The Delannoy number $D(n,k)$ is the number of lattice paths from $(0,0)$ to $(n,k)$
using steps $(1,0), (0,1)$ and $(1,1)$.
Clearly,
$$D(n,k)=D(n-1,k)+D(n-1,k-1)+D(n,k-1),$$
with $D(0,k)=D(k,0)=1$.
It is well known that the Narayana number $NA(n,k)$
counts the number of Dyck paths (using steps $(1,1)$ and $(1,-1)$) from $(0, 0)$ to $(2n, 0)$ with $k$ peaks.
It is also known that $NA_{\langle m\rangle}(n,k)$ counts the number of Dyck paths of semilength $n$ whose last $m$ steps are $(1,-1)$ with $k$ peaks,
see Callan's note in \cite{oeis}.
Brenti \cite[Corollar 5.15]{Bre95} showed that the Delannoy triangle
$D=\left[D(n-k,k)\right]_{n\geq k\geq 0}$ and and the Delannoy square $D^\c=\left[D(n,k)\right]_{n, k\geq 0}$
are TP by means of lattice path techniques.
The following problem naturally arises.

\begin{ques}
Whether the total positivity of Narayana matrices
can also be obtained by a similar combinatorial approach?
\end{ques}

We have seen that the Pascal square has the decomposition $P^\c=PP^T$.
We also have $D^\c=P\diag(1,2,2^2,\ldots)P^T$
since
$$D(n,k)=\sum_{j}2^j\binom{k}{j}\binom{n}{j}$$
(see \cite{BS05} for instance).
A natural problem is to find out the explicit (modified) Choleski decomposition of the Narayana squares $N_A^\c$ and $N_B^\c$.

%
%


Another well-known symmetric triangle is
the Eulerian triangle $A=[A(n,k)]_{n,k\ge 1}$
where $A(n,k)$ is the Eulerian number,
which counts the number of $n$-permutations with exactly $k-1$ excedances.
Brenti \cite[Conjecture 6.10]{Bre96} conjectured that the Eulerian triangle $A$ is TP.
Motivated by the strict total positivity of the Narayana squares, we posed the following conjecture.

\begin{conj}
The Eulerian square $A^\c=[A(n+k,k)]$ is STP.
\end{conj}

\section*{Acknowledgements}

This work was supported by the National Science Foundation of China (Nos. 11231004, 11371078, 11522110).


\end{document}